\documentclass[12pt,a4paper,reqno]{amsart}

\usepackage[T1]{fontenc}
\usepackage{lmodern}
\usepackage{microtype}
\usepackage{fullpage}
\usepackage{amsthm,amsmath,amsfonts,amssymb,mathtools}
\usepackage[colorlinks,citecolor=blue,urlcolor=blue,linkcolor=blue]{hyperref}
\usepackage{graphicx}
\usepackage{tikz}

\numberwithin{equation}{section}

\theoremstyle{plain}
\newtheorem{theorem}{Theorem}[section]
\newtheorem{lemma}[theorem]{Lemma}
\newtheorem{proposition}[theorem]{Proposition}

\theoremstyle{definition}
\newtheorem{definition}[theorem]{Definition}

\newcommand{\Z}{\mathbb Z}
\newcommand{\R}{\mathbb R}
\newcommand{\Pp}{\mathbb P}

\newcommand{\ind}{\mathbf 1}
\newcommand{\cG}{\mathcal G}

\newcommand{\abs}[1]{\lvert#1\rvert}

\newcommand{\dd}{\mathrm d}
\newcommand{\AF}{\mathrm{AF}}
\newcommand{\FM}{\mathrm{FM}}

\title[Phase transition in the antiferromagnetic Ising]{Phase transition in the antiferromagnetic Ising model with external field}

\author{Jianping Jiang}
\address{Yau Mathematical Sciences Center, Tsinghua University, Beijing 100084, China.}
\email{jianpingjiang@tsinghua.edu.cn}
\author{Sike Lang}
\address{Qiuzhen College, Tsinghua University, Beijing 100084, China.}
\email{langsk24@mails.tsinghua.edu.cn}

\begin{document}
	
\begin{abstract}
Consider the nearest-neighbor antiferromagnetic Ising model on $\mathbb{Z}^d$, $d\ge2$, at inverse temperature $\beta$ and in a uniform external field $h$. Let $\beta_c(d)$ be the critical inverse temperature of the zero-field Ising model. We prove that for every $\beta>\beta_c(d)$, the Gibbs state is not unique whenever $|h|$ is sufficiently small.  This extends a previous result of Dobrushin (1968) from sufficiently large $\beta$ to every $\beta>\beta_c(d)$.
\end{abstract}

\maketitle

\section{Introduction and main result}

The nearest-neighbor antiferromagnetic Ising model on a finite set $\Lambda\subset\Z^d$, $d\ge2$, 
at inverse temperature $\beta\ge0$, with boundary condition $\eta\in\{-1,+1\}^{\Z^d}$ and uniform external field $h\in\R$,
 is the probability measure on $\sigma\in\{-1,+1\}^{\Lambda}$ given by
\begin{equation}\label{eq:AF-measure}
 \mu^{\AF,\eta}_{\Lambda,\beta,h}\bigl[\sigma\bigr]
 :=\frac{1}{Z^{\AF,\eta}_{\Lambda,\beta,h}}
 \exp\left[
 -\beta\sum_{\substack{u, v \in \Lambda \\ u \sim v }}\sigma_u\sigma_v
 -\beta\sum_{\substack{u\in\Lambda,\ v\notin\Lambda\\u\sim v}}
          \sigma_u\eta_v
 +h\sum_{u\in\Lambda}\sigma_u
 \right],
\end{equation}
where $u\sim v$ means that $u$ and $v$ are nearest neighbors, and $Z^{\AF,\eta}_{\Lambda,\beta,h}$ is the partition function; each unoriented edge is counted once.  Let $\beta_c(d)$ denote the critical inverse temperature of the zero-field Ising model.

To state the main result, we define the two N\'eel boundary conditions
\begin{equation}\label{eq:Neel-bc}
 \eta^{\mathrm N\pm}_u:=\pm(-1)^{\abs u}, \qquad \abs u:=\sum_{i=1}^d \abs{u_i}, \qquad \forall u\in \mathbb{Z}^d.
\end{equation}
For finite $\Lambda \subset \Z^d$, write
$\mu^{\AF,\mathrm N\pm}_{\Lambda,\beta,h}
:=\mu^{\AF,\eta^{\mathrm N\pm}}_{\Lambda,\beta,h}$.
Let $B_n:=[-n,n]^d\cap\Z^d$.

\begin{theorem}\label{thm:main-AF}
Let $d\ge2$ and $\beta>\beta_c(d)$.  There exist $h_0=h_0(d,\beta)>0$, $c=c(d,\beta)>0$ such that for each $\abs h\le h_0$, 
\begin{equation}\label{eq:finite-gap-AF}
 \mu^{\AF,\mathrm N+}_{B_n,\beta,h}\bigl[\sigma_0=+1\bigr]
 -\mu^{\AF,\mathrm N-}_{B_n,\beta,h}\bigl[\sigma_0=+1\bigr]
 \ge c,
 \qquad \forall n\ge1.
\end{equation}
\end{theorem}

In Section \ref{sec:ferromagnetic-formulation}, we will see that $\mu^{\AF,\mathrm N +}_{B_n,\beta,h}$ and $\mu^{\AF,\mathrm N -}_{B_n,\beta,h}$ converge weakly to the infinite-volume Gibbs states $\mu^{\AF,\mathrm N +}_{\beta,h}$ and $\mu^{\AF,\mathrm N -}_{\beta,h}$ respectively. Theorem \ref{thm:main-AF} implies that $\mu^{\AF,\mathrm N +}_{\beta,h} \neq \mu^{\AF,\mathrm N -}_{\beta,h}$ and thus there are at least two distinct Gibbs states at such $(\beta,h)$. 

If $h=0$, a flip on all odd vertices of $\Z^d$ (i.e., those with $\abs u$ odd) turns the model into the classical zero-field ferromagnetic Ising model. For general $h$, Dobrushin \cite{Dob68} first proved the nonuniqueness of Gibbs states when $\beta$ is very large and $\abs h$ small by a Peierls-type argument.

By \cite[Corollary 1.6]{DSS23} and uniqueness of Gibbs states when $\beta \in [0,\beta_c(d)]$ and $h=0$ \cite{ABF87,Yan52,AF86,ADCS15}, there is always uniqueness whenever $\beta \in [0,\beta_c(d)]$ and $h \in \R$. It is also known that for each $\beta \geq 0$, uniqueness holds whenever $\abs h$ is sufficiently large; see, e.g., \cite[Theorem~5]{CJN18}. The antiferromagnetic Ising model has been a prototypical example and a test case for new ideas and techniques for different uniqueness criteria \cite{DKS85,vdB93,vdBM94,GLM95,FGNPY21}. See Figure~\ref{fig:phase-diagram} for the phase diagram.

\begin{figure}[t]
	\centering
	\begin{tikzpicture}[x=1.02cm,y=0.9cm]
		\fill[orange!10]
		(2.15,0)
		.. controls (2.55,1.10) and (5.35,2.25) .. (8.0,2.75)
		-- (8.0,-2.75)
		.. controls (5.35,-2.25) and (2.55,-1.10) .. (2.15,0)
		-- cycle;
		\fill[blue!70]
		(2.15,0)
		.. controls (2.75,0.12) and (5.25,0.30) .. (8.0,0.40)
		-- (8.0,-0.40)
		.. controls (5.25,-0.30) and (2.75,-0.12) .. (2.15,0)
		-- cycle;
		\draw[densely dashed,very thick]
		(2.15,0) .. controls (2.55,1.10) and (5.35,2.25) .. (8.0,2.75);
		\draw[densely dashed,very thick]
		(2.15,0) .. controls (2.55,-1.10) and (5.35,-2.25) .. (8.0,-2.75);
		\draw[->,thick] (0,0) -- (8.35,0) node[below right] {$\beta$};
		\draw[->,thick] (0,-3.0) -- (0,3.1) node[above left] {$h$};
		\draw (2.15,0.10) -- (2.15,-0.10)
		node[below] {$\beta_c(d)$};
		\node at (5.15,0.65) {nonuniqueness};
		\node at (5.30,2.65) {uniqueness};
		%\node[blue!70!black] at (6.65,-0.55) {proved here};
	\end{tikzpicture}
	\caption{Schematic expected phase diagram for a fixed $d\ge2$.
		The dashed boundary is conjectural.  The blue region indicates
		nonuniqueness proved in Theorem \ref{thm:main-AF}; its width is not quantitative.}
	\label{fig:phase-diagram}
\end{figure}
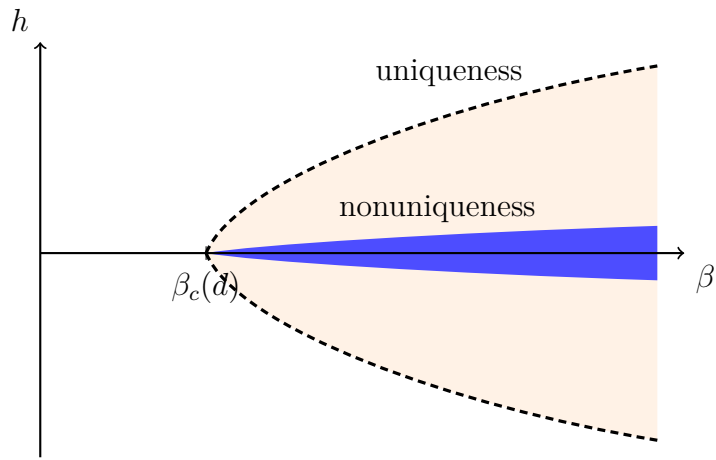

The concept of antiferromagnetic order was first introduced by N\'eel \cite{Nee32} and Landau~\cite{Lan33} in their respective studies of the low-temperature magnetic behavior of certain materials---metals such as Cr and Mn in N\'eel's case, and insulators such as the chlorides of Cr, bivalent Fe, Co and Ni in Landau's case.

Theorem \ref{thm:main-AF} should be contrasted with the corresponding results in the random field Ising model. Imry and Ma \cite{IM75} predicted that quenched disorder destroys long-range order in dimensions $d\leq 2$. For $d=2$, at every positive temperature, Aizenman and Wehr \cite{AW90} proved almost sure uniqueness for a broad class of i.i.d. random fields,  including symmetric bimodal fields $\pm h,  h>0$. For Gaussian disorder, if $d=2$, correlations decay exponentially whenever the disorder strength is nonzero~ \cite{DX21,AHP20}; for $d\geq 3$, a recent breakthrough by Ding, Liu and Xia \cite{DLX24} proved that sufficiently weak disorder preserves phase coexistence whenever $\beta>\beta_c(d)$.

 A main feature of our method is to study the joint spin-bond measure arising from the Edwards-Sokal coupling \cite{ES88,SW87} throughout the whole proof. We first prove that the joint measure is reflection positive; we note that reflection positivity for the spin measure has been used in \cite{FILS80,GNPY16}.  The resulting chessboard estimate, together  with
supercritical FK coarse graining \cite{Pis96,Bod05,CM04}, enables us to prove that the probability of all blocks indexed by a prescribed set $S$ being bad decays exponentially in the size of $S$. This implies that there is an infinite FK mesh cluster, and we use it to define Gibbs states conditioned on the sign of the FK mesh cluster. Using an increasing crossing observable and the
classification of translation-invariant zero-field Gibbs measures \cite{Aiz80,Hig81,Bod06}, we prove continuity of the N\'eel boundary states at zero field. Their distinct zero-field limits then imply nonuniqueness for sufficiently small $\abs h$.

After introducing the notation and preliminary results in Section~\ref{sec:preliminaries}, we establish the main probabilistic
estimate in Section~\ref{sec:chessboard}: a uniform bound on bad-block probabilities in a small staggered field. Section~\ref{sec:infinite-volume} constructs Gibbs states by conditioning on the sign of an infinite FK mesh cluster. In Section~\ref{sec:continuity}, these states give a lower bound on the probability of a positive FK crossing in the maximal state. This estimate identifies its zero-field limit and proves continuity, from which the uniform lower bounds in Theorem \ref{thm:main-AF} follow.

\section{Preliminaries}\label{sec:preliminaries}
For a probability measure $\mu$, write $\langle f\rangle_\mu:=\int f\,\dd\mu$
for expectation and $\mu[A]$ for the probability of an event $A$.
For the Ising measures, we often abbreviate
$\langle f\rangle_{\mu^{\AF,\eta}_{\Lambda,\beta,h}}$
as $\langle f\rangle^{\AF,\eta}_{\Lambda,\beta,h}$,
with the same convention for the ferromagnetic and infinite-volume measures.   $\Lambda \Subset \Z^d$ means $\Lambda$ is a finite subset, and 
\[E(\Lambda):=\{uv: u \sim v, u, v \in \Lambda\}, \qquad \bar{E}(\Lambda):=\{uv: u \sim v, \{u,v\} \cap \Lambda \neq \emptyset\}\]
are, respectively, the sets of nearest-neighbor edges in $\Lambda$ and those with nonempty intersection with $\Lambda$. For probability measures on the configuration spaces considered in the paper, a sequence  $(\mu_n)$ converges weakly to $\mu$ if 
\[ \lim_{n \rightarrow \infty} \mu_n[A]=\mu[A] \text{ for each local event } A. \]
See \cite[Section 6.4.2]{FV18} for more details.

\subsection{The ferromagnetic formulation}\label{sec:ferromagnetic-formulation}
For each $\sigma \in \{-1,+1\}^{\Z^d}$, we define $\widehat{\sigma}$ by $\widehat\sigma_u=(-1)^{\abs u}\sigma_u$ (i.e., flip all spins on odd vertices). Using this transformation, we turn the antiferromagnetic Ising model defined in \eqref{eq:AF-measure} into the following ferromagnetic Ising model with staggered field.
\begin{equation}\label{eq:FM-measure}
 \mu^{\FM,\eta}_{\Lambda,\beta,h}\bigl[\sigma\bigr]
 :=\frac{1}{Z^{\FM,\eta}_{\Lambda,\beta,h}}
 \exp\left[
 \beta\sum_{uv\in E(\Lambda)}\sigma_u\sigma_v
 +\beta\sum_{\substack{u\in\Lambda,\ v\notin\Lambda\\u\sim v}}
          \sigma_u\eta_v
 +h\sum_{u\in\Lambda}(-1)^{\abs u}\sigma_u
 \right].
\end{equation}
More precisely, we have
\[Z^{\AF,\eta}_{\Lambda,\beta,h}=Z^{\FM,\widehat{\eta}}_{\Lambda,\beta,h}, \qquad \mu^{\AF,\eta}_{\Lambda,\beta,h}\bigl[\sigma\bigr]=\mu^{\FM,\widehat{\eta}}_{\Lambda,\beta,h}\bigl[\widehat{\sigma}\bigr], \qquad \forall \sigma \in \{-1,+1\}^{\Lambda}.\]
In particular,  the N\'eel labels $\mathrm N\pm$ in the antiferromagnetic model correspond to the constant boundary labels $\pm$ in the ferromagnetic model.

By monotonicity in the boundary condition, the weak limits
\begin{equation}\label{eq:max-min}
 \mu^{\FM,+}_{\beta,h}
 :=\lim_{\Lambda\uparrow\Z^d}\mu^{\FM,+}_{\Lambda,\beta,h},
 \qquad
 \mu^{\FM,-}_{\beta,h}
 :=\lim_{\Lambda\uparrow\Z^d}\mu^{\FM,-}_{\Lambda,\beta,h}
\end{equation}
exist.  Moreover, both $ \mu^{\FM,+}_{\beta,h}$ and $\mu^{\FM,-}_{\beta,h}$ are extremal Gibbs states, translation invariant under the even sublattice $\{u \in \Z^d: \abs u \text{ even}\}$, and all other Gibbs measures lie in between $ \mu^{\FM,+}_{\beta,h}$ and $\mu^{\FM,-}_{\beta,h}$ in the FKG sense. The limits \eqref{eq:max-min} and their properties are standard for the ferromagnetic Ising model with uniform external field (see, e.g., \cite[Theorem 3.17 and Lemma 6.65]{FV18}); their proofs adapt readily to the case of a staggered field. 

The transformation $\sigma_u\mapsto-\sigma_{u+e_i}$ exchanges the
maximal and minimal states. Consequently,
\begin{equation}\label{eq:m-symmetry}
 \langle \sigma_0 \rangle^{\FM,\pm}_{\beta,h}=- \langle \sigma_{e_1} \rangle^{\FM,\mp}_{\beta,h}.
\end{equation}

To prove Theorem~\ref{thm:main-AF}, it suffices to establish the following equivalent ferromagnetic statement.

\begin{proposition}\label{prop:main-FM}
For $d\ge2$ and $\beta>\beta_c(d)$, there exist $h_0=h_0(d,\beta)>0,c=c(d,\beta)>0$ such that for each $|h|\leq h_0$,
\begin{equation}\label{eq:FM-main-gap}
\mu^{\FM,+}_{B_n,\beta,h}\bigl[\sigma_0=+1\bigr]
-\mu^{\FM,-}_{B_n,\beta,h}\bigl[\sigma_0=+1\bigr]
\ge c,
\qquad \forall n\ge1.
\end{equation}
\end{proposition}
In contrast, it is well known that the ferromagnetic Ising model with uniform external field $h\neq 0$ has a unique Gibbs state regardless of the value of $\beta$; see, e.g., \cite[Section~3.7.4]{FV18}.

\subsection{The joint spin-bond measure}
The Edwards-Sokal coupling \cite{ES88,SW87} is a probability measure on a common probability space for the Ising model and FK percolation (also called the random-cluster model with cluster weight $2$). We refer to \cite{CJN18} for the coupling with vertex-dependent external fields, constructed by adding ghost vertices; in this paper, we use its marginal on $\Z^d$. For a finite $\Lambda \Subset \Z^d$ and boundary condition $\eta \in \{-1,+1\}^{\Z^d}$, the joint spin-bond measure is
\begin{equation}\label{eq:ES-measure-general}
	\Pp_{\Lambda,\beta,h}^{\eta}\bigl[\sigma,\omega\bigr] \propto \exp[h \sum_{u\in \Lambda} (-1)^{\abs u} \sigma_u] \prod_{e=uv \in \bar{E}(\Lambda)} [(1-p) \delta_{\omega_e,0} + p \delta_{\omega_e,1} \delta_{\sigma_u,\sigma_v}],
\end{equation}
where $(\sigma,\omega) \in \{-1,+1\}^{\Lambda} \times \{0,1\}^{\bar{E}(\Lambda)}$, $p:=1-e^{-2\beta}$, $\sigma_v=\eta_v$ for each $v\notin \Lambda$, and $\delta_{i,j}$ is the Kronecker delta. The spin marginal of $\Pp_{\Lambda,\beta,h}^{\eta}$ is $ \mu^{\FM,\eta}_{\Lambda,\beta,h}$, defined in \eqref{eq:FM-measure}. The bond marginal of the joint measure depends on $h$; in the torus case \eqref{eq:ES-measure}, it reduces at $h=0$ to the FK percolation measure. Conditional on the spins, bonds joining equal spins are independently open with probability $p$, and all other bonds are closed.

Fix an even integer $L$ and let $\mathbb T_L:=(\Z/L\Z)^d$.  The joint spin-bond measure on $\mathbb{T}_L$ is simply written as
\begin{equation}\label{eq:ES-measure}
 \Pp_{L,\beta,h}\bigl[\sigma,\omega\bigr]
\propto \exp[h \sum_{u\in \mathbb{T}_L} (-1)^{\abs u} \sigma_u] \prod_{e=uv \in E(\mathbb{T}_L)} [(1-p) \delta_{\omega_e,0} + p \delta_{\omega_e,1} \delta_{\sigma_u,\sigma_v}],
\end{equation}
where $(\sigma,\omega) \in \{-1,+1\}^{\mathbb{T}_L} \times \{0,1\}^{E(\mathbb{T}_L)}$.

\subsection{The zero-field coarse-graining input}

Choose an odd block scale $q$.  For $v\in\Z^d$, set
\begin{equation}\label{eq:blocks}
 Q_v:=qv+[-q,q]^d \cap \Z^d,
 \qquad
 R_v:=qv+[-2q,2q]^d \cap \Z^d.
\end{equation}

\begin{definition}\label{def:good-block}
A block $Q_v$ is \emph{good} if the restriction of the FK configuration to $E(Q_v)$ contains a cluster touching all $2d$ faces of $Q_v$, and every open path in $Q_v$ with $\ell^\infty$-diameter at least $q$ belongs to that cluster.  Otherwise, $Q_v$ is \emph{bad}.
\end{definition}

Write $\phi^\xi_{\Lambda,\beta}$ for the zero-field FK percolation measure on a finite set $\Lambda \Subset \Z^d$
with edge parameter $p=1-e^{-2\beta}$ and boundary condition $\xi$ (i.e., a partition of
the boundary vertices specifying which are wired).
The zero-field input is the following coarse-graining estimate.

\begin{proposition}\label{prop:zero-field-block}
Let $d\ge2$ and $\beta>\beta_c(d)$.  There exist $c_0=c_0(d,\beta) \in (0,\infty)$ and $q_0=q_0(d,\beta) \in (0,\infty)$ such that, for every integer $q\ge q_0$ and every FK boundary condition $\xi$ on $R_v$,
\begin{equation}\label{eq:zero-field-block}
 \phi^{\xi}_{R_v,\beta}\bigl[Q_v\text{ is bad}\bigr]\le e^{-c_0q}.
\end{equation}
\end{proposition}

For $d\ge3$, Proposition \ref{prop:zero-field-block} is a consequence of Pisztora's coarse-graining estimate 
\cite[Theorem 3.1]{Pis96} and Bodineau's equality of the slab and bulk critical points \cite{Bod05}.  For $d=2$, it follows from Couronn\'e and Messikh \cite[Theorem 9]{CM04}.

\section{The chessboard estimate}\label{sec:chessboard}

In this section, we bound the probability that all blocks indexed by an arbitrary set $S$ of coarse vertices are bad under a small staggered field. A direct comparison with the zero-field measure yields a factor exponential in the volume of the torus and therefore gives no useful uniform estimate.  By applying the chessboard estimate to the joint spin-bond measure, we obtain instead a cost for each prescribed bad block.  Together with Proposition~\ref{prop:zero-field-block}, this gives exponential decay in |S|.

\subsection{Reflection positivity}

For even $L\ge4$, $1\le i\le d$ and integers $0\le b<L/2$, define
the coordinate reflection
\begin{equation}\label{eq:coordinate-reflection}
 \vartheta_{i,b}u:=u+2(b-u_i)e_i\pmod L.
\end{equation}
It also acts on spins and bonds by
\[
 (\vartheta_{i,b}\sigma)_u:=\sigma_{\vartheta_{i,b}u},
 \qquad
 (\vartheta_{i,b}\omega)_{uv}
 :=\omega_{\vartheta_{i,b}u,\vartheta_{i,b}v}.
\]
Figure~\ref{fig:torus-reflection} illustrates this reflection in
dimension two.

\begin{figure}[tbp]
\centering
\includegraphics[width=0.75\textwidth]{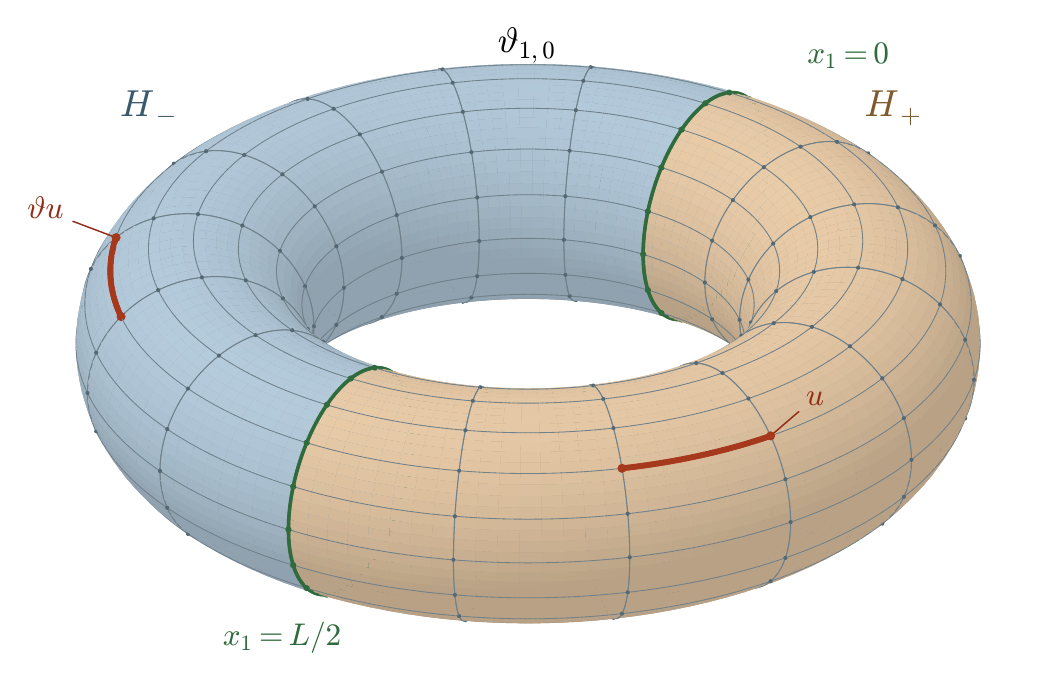}
\caption{The two-dimensional torus embedded in three-dimensional space.
The reflection $\vartheta_{1,0}$ exchanges the two shaded half-tori
and fixes the two circles $x_1=0$ and $x_1=L/2$.
Green vertices and edges belong to the reflection plane and generate
$\mathcal F_0$.
The marked vertex and red edge are shown together with their
reflected images. }
\label{fig:torus-reflection}
\end{figure}

To apply the standard formulation of reflection positivity, encode
the joint configuration as site variables
$X\in\{-1,0,1\}^{\mathbb T_{2L}}$ by
\begin{equation}\label{eq:site-encoding}
 X_{2u}:=\sigma_u,\qquad
 X_{2u+e_j}:=\omega_{u,u+e_j},
\end{equation}
and set $X_z=0$ at all remaining sites.  Denote the induced measure by $\widetilde{\Pp}_{L,\beta,h}$.  The original spins and bonds under $\Pp_{L,\beta,h}$ can be recovered from $X$.  Moreover, the counterpart of $\vartheta_{i,b}(\sigma,\omega)$ is $\widetilde\vartheta_{i,b}X$, where
\[
 \widetilde\vartheta_{i,b}z
 :=z+2(2b-z_i)e_i\pmod{2L}.
\]

The reflection plane (say $\widetilde{P}$) corresponding to $\widetilde\vartheta_{i,b}$ divides $\mathbb{T}_{2L}$ into two halves, $\mathbb{T}_{2L}^+$ and $\mathbb{T}_{2L}^-$, such that $\mathbb{T}_{2L}^+ \cap \mathbb{T}_{2L}^-$ is $\widetilde{P} \cap \mathbb{T}_{2L}$. Let $\mathcal F_+$ be the $\sigma$-algebra generated by $X$ restricted to $\mathbb{T}_{2L}^+$. By
\cite[Definition~5.1]{Bis09}, reflection positivity means that
\begin{equation}\label{eq:reflection-positivity}
 \begin{aligned}
 \langle F\,\widetilde\vartheta_{i,b}G\rangle_{\widetilde{\Pp}_{L,\beta,h}}
 &=
 \langle G\,\widetilde\vartheta_{i,b}F\rangle_{\widetilde{\Pp}_{L,\beta,h}},
 \\
 \langle F\,\widetilde\vartheta_{i,b}F\rangle_{\widetilde{\Pp}_{L,\beta,h}}
 &\ge0
 \end{aligned}
\end{equation}
for all bounded real-valued $\mathcal F_+$-measurable functions
$F,G$, where
$\widetilde\vartheta_{i,b}F:=F\circ\widetilde\vartheta_{i,b}$.

\begin{lemma}\label{lem:reflection-positivity}
For every $i,b$ as above, $\widetilde{\Pp}_{L,\beta,h}$ is reflection
positive with respect to $\widetilde\vartheta_{i,b}$.
\end{lemma}

\begin{proof}
Fix $i,b$, and abbreviate
$\widetilde\vartheta=\widetilde\vartheta_{i,b}$ and
$\widetilde{\Pp}=\widetilde{\Pp}_{L,\beta,h}$.  Let $\mathcal{F}_0$ be the $\sigma$-algebra generated by $X_z$'s with $z$ in the reflection plane: 
\[
 \mathcal F_0
 :=\sigma\bigl(X_z:z_i=2b\text{ or }z_i=2b+L\bigr).
\]
Apart from deterministic variables, this records precisely the spins on the original reflection plane and the bonds lying entirely within it.

Since the reflecting planes pass through lattice sites, \eqref{eq:ES-measure} implies that the two half-tori are conditionally
independent given $\mathcal F_0$.  Since $\vartheta_{i,b}$ preserves vertex parity, 
\[\Pp_{L,\beta,h}\bigl[\sigma,\omega\bigr]=\Pp_{L,\beta,h}\bigl[\vartheta_{i,b}\sigma,\vartheta_{i,b}\omega\bigr].\]
The reflection fixes all $X_z$'s generating
$\mathcal F_0$, so it exchanges the two conditional laws.  Hence
\[
 \left\langle F\,\widetilde\vartheta G\mid\mathcal F_0\right\rangle_{\widetilde{\Pp}}
 =
 \left\langle F\mid\mathcal F_0\right\rangle_{\widetilde{\Pp}}
 \left\langle G\mid\mathcal F_0\right\rangle_{\widetilde{\Pp}}.
\]
Averaging over the joint law of the variables generating
$\mathcal F_0$ gives
\begin{equation}\label{eq:rp-conditional-product}
 \langle F\,\widetilde\vartheta G\rangle_{\widetilde{\Pp}}
 =
 \left\langle
 \left\langle F\mid\mathcal F_0\right\rangle_{\widetilde{\Pp}}
 \left\langle G\mid\mathcal F_0\right\rangle_{\widetilde{\Pp}}
 \right\rangle_{\widetilde{\Pp}}.
\end{equation}
This expression is symmetric in $F,G$, and taking $G=F$ yields
\[
 \langle F\,\widetilde\vartheta F\rangle_{\widetilde{\Pp}}
 =
 \left\langle
 \left(\left\langle F\mid\mathcal F_0\right\rangle_{\widetilde{\Pp}}\right)^2
 \right\rangle_{\widetilde{\Pp}}
 \ge0.
\]
Thus both conditions for reflection positivity hold.
\end{proof}

We may therefore apply the chessboard estimate \cite[Theorem~5.8]{Bis09} to $\widetilde{\Pp}_{L,\beta,h}$ and thus $\Pp_{L,\beta,h}$.

\subsection{Bad blocks in a staggered field}

Fix $d\ge2$ and $\beta>\beta_c(d)$.  Let $q\ge q_0$ be odd, where $q_0$ is as in Proposition~\ref{prop:zero-field-block}, and take $L=q2^{k+1}$ for some integer $k\ge2$.  We index the coarse vertices by $\mathbb T_{L/q}$ and regard the blocks $Q_v,R_v$ as subsets of $\mathbb T_L$.  Divide the coarse vertices into the $2^d$ classes
\begin{equation}\label{eq:classes}
 \mathcal A_r:=\{v\in\mathbb T_{L/q}:v\equiv r\pmod 2\},
 \qquad r\in\{0,1\}^d.
\end{equation}
For each $r$, the blocks $\{Q_v:v\in\mathcal A_r\}$ have side
length $2q$ and tile the torus with shared faces.  Let
\begin{equation}\label{eq:number-blocks}
 M:=\abs{\mathcal A_r}=\left(\frac{L}{2q}\right)^d.
\end{equation}

We first estimate the probability that every block in one such
tiling is bad at zero field.

\begin{lemma}\label{lem:zero-field-all-bad}
Recall $c_0$ and $q_0$ from Proposition~\ref{prop:zero-field-block}.  Let $c_1:=c_0/4^d$. Then, for every $r\in\{0,1\}^d$,
\begin{equation}\label{eq:zero-field-all-bad}
 \Pp_{L,\beta,0}\left[Q_v\text{ is bad for every }v\in\mathcal A_r\right]
 \le e^{-c_1qM}, \qquad \forall q\geq q_0, L=q 2^{k+1} \text{ with }k\geq 2. 
\end{equation}
\end{lemma}

\begin{proof}
Set
\[
 \mathcal I_r:=\{v\in\mathbb T_{L/q}:v\equiv r\pmod 8\}.
\]
Since $8$ divides $L/q$, we have $\mathcal I_r\subset\mathcal A_r$
and $\abs{\mathcal I_r}=M/4^d$.  The enlarged blocks
$\{R_v:v\in\mathcal I_r\}$ are pairwise disjoint.

At zero field, the bond marginal of $\Pp_{L,\beta,0}$ is FK percolation. Conditional on the bonds outside $E(R_v)$, its law inside $R_v$ is $\phi^{\xi}_{R_v,\beta}$ for the boundary condition $\xi$ induced by the exterior connections.  Proposition~\ref{prop:zero-field-block}
therefore gives
\[
 \Pp_{L,\beta,0}\left[
 Q_v\text{ is bad}\,\middle|\,
 \omega_{E(\mathbb T_L)\setminus E(R_v)}
 \right]\le e^{-c_0q}
 \qquad\text{almost surely}.
\]
For distinct $v,w\in\mathcal I_r$, the event that $Q_w$ is bad is measurable
with respect to the bonds outside $E(R_v)$.  Iterating the
conditional bound over $\mathcal I_r$ yields
\begin{align*}
 \Pp_{L,\beta,0}\left[Q_v\text{ is bad for every }v\in\mathcal A_r\right]
 &\le
 \Pp_{L,\beta,0}\left[Q_v\text{ is bad for every }v\in\mathcal I_r\right]
 \\
 &\le e^{-c_0q\abs{\mathcal I_r}}
 =e^{-c_1qM}.
\end{align*}
This completes the proof of the lemma.
\end{proof}

We next compare the measures with and without staggered field.  By \eqref{eq:ES-measure},
\[
 \frac{\dd\Pp_{L,\beta,h}}{\dd\Pp_{L,\beta,0}}(\sigma,\omega)
 =
 \frac{\exp[h\sum_{u\in\mathbb T_L}(-1)^{\abs u}\sigma_u]}
 {\langle \exp[h\sum_{u\in\mathbb T_L}(-1)^{\abs u}\sigma_u]
 \rangle_{\Pp_{L,\beta,0}}}
 \le e^{2\abs hL^d},
\]
where the bound follows from $\abs{\sum_{u\in\mathbb T_L}(-1)^{\abs u}\sigma_u}\le L^d$. Consequently, for each event $A$,
\begin{equation}\label{eq:global-tilt}
 \Pp_{L,\beta,h}\bigl[A\bigr]
 \le e^{2\abs hL^d}\Pp_{L,\beta,0}\bigl[A\bigr].
\end{equation}

For each $v\in\mathcal A_r$, the event $\{Q_v \text{ is bad}\}$ depends only on the bonds in $E(Q_v)$.  Reflection through a shared face maps the bad-block event to the corresponding event in the reflected block, and Lemma~\ref{lem:reflection-positivity} gives reflection positivity for all such reflections.  Hence the chessboard estimate
\cite[Theorem~5.8]{Bis09} gives, for every
$S\subset\mathcal A_r$,
\begin{equation*}
 \Pp_{L,\beta,h}\left[Q_v\text{ is bad for every }v\in S\right]
 \le
 \Pp_{L,\beta,h}\left[
 Q_v\text{ is bad for every }v\in\mathcal A_r
 \right]^{\abs S\left(\frac{2q}{L}\right)^d}.
\end{equation*}
Combining this with \eqref{eq:global-tilt} and
Lemma~\ref{lem:zero-field-all-bad} yields
\begin{equation}\label{eq:chessboard-main}
 \Pp_{L,\beta,h}\left[Q_v\text{ is bad for every }v\in S\right]
 \le \exp[(2\abs h(2q)^d-c_1q)\abs S].
\end{equation}

This gives a uniform estimate for an arbitrary collection of blocks.

\begin{proposition}\label{prop:uniform-bad-block}
Recall $q_0$ and $c_1$ from Lemma \ref{lem:zero-field-all-bad}. For every odd $q\ge q_0$, set
\begin{equation}\label{eq:h0-choice}
 h_q:=\frac{c_1q}{4(2q)^d},\qquad
 r_q:=\exp\left[-\frac{c_1q}{2^{d+1}}\right].
\end{equation}
Then, for every $\abs h\le h_q$ and every
$S\subset\mathbb T_{L/q}$,
\begin{equation}\label{eq:uniform-bad-block}
 \Pp_{L,\beta,h}\left[Q_v\text{ is bad for every }v\in S\right]
 \le r_q^{\abs S},
\end{equation}
uniformly over $L=q2^{k+1}$ with integer $k\ge2$.  In particular,
$r_q\downarrow0$ as $q\to\infty$ through odd integers.
\end{proposition}

\begin{proof}
Choose $r$ such that $\abs{S\cap\mathcal A_r}\ge\abs S/2^d$.  Since $2\abs h(2q)^d\le c_1q/2$, \eqref{eq:chessboard-main} gives
\[
 \begin{aligned}
 \Pp_{L,\beta,h}\left[Q_v\text{ is bad for every }v\in S\right]
 &\le
 \Pp_{L,\beta,h}\left[Q_v\text{ is bad for every }v\in S\cap\mathcal A_r\right]
 \\
 &\le
 \exp\left[-\frac{c_1q}{2}\abs{S\cap\mathcal A_r}\right]
 \le r_q^{\abs S}.
 \end{aligned}
\]
This completes the proof of the proposition.
\end{proof}

\section{Construction of conditioned Gibbs states}\label{sec:infinite-volume}

Let $q$ be a sufficiently large odd block scale and $\abs h\le h_q$. Taking $L\to\infty$ through the tori of Section~\ref{sec:chessboard}, let $\Pp_{\beta,h}$ be a subsequential weak limit of $\Pp_{L,\beta,h}$; see  \cite[Theorem~6.24]{FV18} for the justification of such a limit. We first use Proposition~\ref{prop:uniform-bad-block} to obtain an infinite FK mesh cluster (i.e., an infinite FK open cluster arising from an infinite good-block component). Conditioning on its sign gives Gibbs states whose crossing probabilities will be used in Section~\ref{sec:continuity}. The construction applies to every sufficiently large odd $q$.

\subsection{The infinite component of good blocks}

Two vertices of $\Z^d$ are called $*$-neighbors if their
$\ell^\infty$-distance is one.  We prove the following general
dependent-percolation lemma by a Peierls argument.

\begin{lemma}\label{lem:good-geometry}
Let $d\ge2$ and $\mu$ be a probability measure on $\{0,1\}^{\Z^d}$.  Call a vertex \emph{open} when $\eta_v=1$ and \emph{closed} when $\eta_v=0$.  Suppose that, for some $r>0$,
\begin{equation}\label{eq:abstract-bad-bound}
 \mu\bigl[\eta_v=0\text{ for every }v\in A\bigr]\le r^{\abs A}, \qquad \forall A\Subset\Z^d.
\end{equation}
Set $\kappa_d=(3^d-1)^2$.
If $r<\kappa_d^{-1}$, then the open vertices have,
$\mu$-almost surely, a unique infinite nearest-neighbor component
$\mathcal C_\infty$.  Moreover, for every $v\in\Z^d$,
\begin{equation}\label{eq:good-density}
 \mu\bigl[v\notin\mathcal C_\infty\bigr]
 \le r+\frac{\kappa_dr}{(1-\kappa_dr)^2}
 =:\delta_d(r).
\end{equation}
In particular, $\delta_d(r)\to0$ as $r\downarrow0$.
\end{lemma}

\begin{proof}
For a finite connected set $A$, let $\partial_{\mathrm{ex}}A$ be the
vertices adjacent to $A$ in the infinite component of $\Z^d\setminus A$.
This boundary set is $*$-connected by \cite[Lemma~2]{Tim13}.

Put $\rho=\kappa_dr<1$.  Every $*$-connected set containing $v$ with at
least $n$ vertices contains a tree with $n$ vertices containing $v$.
A traversal of this tree is a $*$-nearest-neighbor walk of length
$2(n-1)$, and there are at most
$(3^d-1)^{2(n-1)}\le\kappa_d^n$ such walks.  Hence
\begin{equation}\label{eq:large-bad-animal}
 \mu[\exists A \text{ s.t. } v\in A, \ A\text{ is $*$-connected and entirely closed}, \abs A\ge n]\le\rho^n.
\end{equation}
Letting $n\to\infty$ in \eqref{eq:large-bad-animal} and then taking a countable union over $v$ imply that almost surely there is no infinite closed $*$-component.
Call a finite set a \emph{separator of $v$} if it separates $v$ from infinity.
A $*$-connected separator of size $n$ meets
$\{v+je_1:1\le j\le n\}$.  The same count therefore gives
\begin{equation}\label{eq:bad-separator-count}
 \mu[\exists\text{ a closed $*$-connected separator }A\text{ of }v \text{ s.t. }
             \abs A=n]\le n\rho^n.
\end{equation}

For existence of $\mathcal{C}_{\infty}$, suppose all open components are finite.  Making each vertex in a finite
box $B_N(v):=v+[-N,N]^d \cap \Z^d$ open cannot create an infinite component, since
only finitely many original components can be joined.  The open component
containing $B_N(v)$ (assuming each vertex in $B_N(v)$ is open) thus has a finite, closed $*$-connected exterior boundary,
which lies outside $B_N(v)$ and is closed also in the original configuration.
These boundaries give separators of $v$ of arbitrarily large size,
an event of probability zero since \eqref{eq:bad-separator-count} implies that $\sum_{n\ge N}n\rho^n\to0$ as $N\rightarrow \infty$.

For uniqueness, suppose $\mathcal{C},\mathcal{C}'$ are distinct infinite open components,
and let $U$ be the component of $\Z^d\setminus \mathcal{C}$ containing $\mathcal{C}'$.
Again by \cite[Lemma~2]{Tim13}, the vertices in $U$ adjacent to $\mathcal{C}$
form a closed $*$-connected set.  This set separates $\mathcal{C}$ from $\mathcal{C}'$
and is therefore infinite; this is because deleting finitely many vertices from
$\Z^d$, $d\ge2$, leaves only one infinite component.  This contradicts the statement right after
\eqref{eq:large-bad-animal}.

Finally, if $v\notin\mathcal C_\infty$, then $v$ is either closed or its finite
open component has a closed exterior boundary separating $v$ from infinity.
Summing \eqref{eq:bad-separator-count} over all sizes gives
\begin{equation}\label{eq:delta-explicit}
 \mu[v\notin\mathcal C_\infty]\le r+\sum_{m\ge1}m\rho^m
 =r+\frac{\rho}{(1-\rho)^2}=\delta_d(r).
\end{equation}
This completes the proof of the lemma.
\end{proof}

We now apply Lemma~\ref{lem:good-geometry} to the site process
$\eta_v=\ind_{\{Q_v\text{ is good}\}}$.  By
Proposition~\ref{prop:uniform-bad-block} and weak convergence,
its law satisfies \eqref{eq:abstract-bad-bound} with $r=r_q$.  Taking $q$ large
enough that $\kappa_dr_q<1$, denote its unique infinite open component by $\cG_\infty$. Then
\begin{equation}\label{eq:delta-choice}
\Pp_{\beta,h}[v \notin \cG_{\infty}] \leq \delta_q:=\delta_d(r_q)\rightarrow0
 \qquad\text{as }q\to\infty.
\end{equation}
By Definition~\ref{def:good-block}, the crossing clusters in neighboring
good blocks intersect.  Their union over $v\in\cG_\infty$ is therefore
an infinite connected FK mesh cluster, whose spins have a common sign. For $s\in \{-1,+1\}$, we call a block \emph{$s$-good} if it is good and each spin on its crossing cluster equals $s$; define
\begin{equation}\label{eq:mesh-sign}
A_s := \{\exists \text{ an infinite nearest-neighbor component of } s\text{-good blocks}\}.
\end{equation}
In particular, $A_s$ occurs if the FK mesh cluster corresponding to $\mathcal{G}_{\infty}$ takes the sign $s$.
%Of course, two blocks $Q_u$ and $Q_v$ with $u,v \in \Z^d$ are nearest neighbor if and only if $u \sim v$.

\subsection{Conditioning on the mesh sign}

\begin{lemma}\label{lem:tail-sign}
For each $s\in \{-1,+1\}$,  $A_s$ is a tail event. Furthermore,
\begin{equation}\label{eq:sign-half}
 \Pp_{\beta,h}[A_+]=\Pp_{\beta,h}[A_-]=1/2.
\end{equation}
\end{lemma}

\begin{proof}
A finite change of spins and bonds changes the $s$-good status of only finitely many blocks. Removing finitely many vertices from an infinite component of a locally finite graph leaves an infinite component. Adding finitely many vertices cannot join infinitely many previously finite components. Hence $A_s$ is invariant under finite coordinate changes and is a tail event.

Since $q$ is odd, the transformation
\begin{equation}\label{eq:combined-symmetry}
 (\mathcal T\sigma)_u=-\sigma_{u+qe_1},
 \qquad
 (\mathcal T\omega)_{uv}=\omega_{u+qe_1,v+qe_1},
\end{equation}
preserves the torus measure defined in \eqref{eq:ES-measure}, and hence its limit $\Pp_{\beta,h}$.  The infinite component of good blocks is unique and the corresponding FK mesh cluster has a common sign, so $\Pp_{\beta,h}[A_+]+\Pp_{\beta,h}[A_-]=1$. Since $\mathcal{T}$ shifts the coarse lattice by $e_1$ and reverses the sign of the FK mesh cluster, $\Pp_{\beta,h}[A_+]=\Pp_{\beta,h}[A_-]$. This proves \eqref{eq:sign-half}.
\end{proof}

Define
\begin{equation}\label{eq:conditioned-joint-states}
 \widehat\Pp^{\,s}_{\beta,h}
 :=\Pp_{\beta,h}[\,\cdot\mid A_s],
 \qquad s\in\{-1,+1\},
\end{equation}
and let $\widehat\mu^s_{\beta,h}$ denote its spin marginal. We have suppressed the dependence of $ \widehat\Pp^{\,s}_{\beta,h}$ on $q$.

\begin{proposition}\label{prop:conditioned-DLR}
Let $\Lambda \Subset \Z^d$. Under $\widehat\Pp^{\,s}_{\beta,h}$, conditional on the $\sigma$-algebra
$\mathcal F_{\Lambda^c}$ generated by spins in $\Lambda^c$ and bonds in $E(\Lambda^c)$, the distribution of
$(\sigma_\Lambda,\omega_{\bar{E}(\Lambda)})$ is the joint spin-bond measure $\Pp^{\eta}_{\Lambda,\beta,h}$ defined in \eqref{eq:ES-measure-general} where $\eta_v=\sigma_v$ for each $v\notin \Lambda$. 

The measure $\widehat\mu^s_{\beta,h}$ is a Gibbs measure for the Ising model defined in \eqref{eq:FM-measure}. Fix a finite bond set $F$. Under $\widehat\Pp^{\,s}_{\beta,h}$, conditional on all spins, the bonds in $F$ are independent and
\[
 \widehat\Pp^{\,s}_{\beta,h}
 [\omega_{uv}=1\mid\sigma]
 =p\ind_{\{\sigma_u=\sigma_v\}},
 \qquad uv \in F.
\]
\end{proposition}

\begin{proof}
For a full configuration $(\sigma, \omega)$, let $\pi_{\Lambda}$ be the kernel that samples $(\sigma_{\Lambda}, \omega_{\bar{E}(\Lambda)})$ according to~\eqref{eq:ES-measure-general}, with exterior spins and bonds fixed. 
Then an easy calculation shows
\[\Pp_{L,\beta,h}[B]=\sum_{\eta \in \{-1,+1\}^{\mathbb{T}_L\setminus \Lambda}} \sum_{\xi \in \{0,1\}^{E(\mathbb{T}_L \setminus \Lambda)}} \pi_{\Lambda} [B \mid \eta, \xi] \Pp_{L,\beta,h}[\eta,\xi],\quad \forall \Lambda \subset \mathbb{T}_L, \text{ local event }B.\]	

Passing to the chosen subsequential limit,
\[\Pp_{\beta,h}[B]= \int \pi_{\Lambda} [B \mid \eta, \xi] \Pp_{\beta,h}[d\eta, d\xi], \qquad \forall \Lambda \Subset \Z^d, \text{ local event }B.\]
Therefore,
\begin{equation*}
	\Pp_{\beta,h}[A \cap B]= \int_A \pi_{\Lambda} [B \mid \eta, \xi] \Pp_{\beta,h}[d\eta, d\xi], \qquad \forall \text{ local event }B, \ A \in \mathcal{F}_{\Lambda^c}, \ \Lambda \Subset \Z^d.
\end{equation*}
These are the DLR equations for the joint measure. Since $A_s$ is a tail event by Lemma~\ref{lem:tail-sign},
\begin{equation}\label{eq:DLRconditional}
	\widehat\Pp^{\,s}_{\beta,h}[B]=\frac{\Pp_{\beta,h}[B \cap A_s]}{\Pp_{\beta,h}[A_s]}=\frac{\int_{A_s}  \pi_{\Lambda} [B \mid \eta, \xi] \Pp_{\beta,h}[d\eta, d\xi]}{\Pp_{\beta,h}[A_s]}=\int \pi_{\Lambda} [B \mid \eta, \xi]  \widehat\Pp^{\,s}_{\beta,h}[d\eta, d\xi]
\end{equation}
for each local event $B$ and $\Lambda \Subset \Z^d$. These are the DLR equations for $\widehat\Pp^{\,s}_{\beta,h}$, which completes the proof of the first statement.

If $B$ in \eqref{eq:DLRconditional} is a local event for the Ising model, we get
\[\widehat\mu^s_{\beta,h}[B]=\int \mu^{\FM,\eta}_{\Lambda,\beta,h}[B] \widehat\mu^s_{\beta,h}[d\eta], \qquad \forall \text{ local Ising event }B, \ \Lambda \Subset \Z^d;\]
here $\mu^{\FM,\eta}_{\Lambda,\beta,h}$ is understood as its corresponding kernel if $B$ is not supported in $\Lambda$. These are the DLR equations for the Ising model.

Finally, \eqref{eq:DLRconditional} implies that, given $\mathcal{F}_{\Lambda^c}$ and the $\sigma$-algebra generated by spins in $\Lambda$, the bonds in $\bar{E}(\Lambda)$ are independent with opening probabilities $p \ind_{\{\sigma_u=\sigma_v\}}$. For any finite bond set $F$, choosing $\Lambda$ containing its endpoints and applying the tower property complete the proof of the last statement.
\end{proof}

\section{Continuity at zero field}\label{sec:continuity}

Fix $d\ge2$ and $\beta>\beta_c(d)$, and write $\mu_h^\pm:=\mu^{\FM,\pm}_{\beta,h}$. In this section, we prove
$\mu_h^\pm\Rightarrow\mu_0^\pm$ as $h\to0$ by using a local observable to exclude any minus phase component from a limit of $\mu_h^+$. Continuity then gives nonuniqueness and the uniform lower bounds in Theorem~\ref{thm:main-AF}.

\subsection{A positive FK crossing observable}
Let $C_q^+$ be the event that $Q_0$ contains an open path joining its two faces perpendicular to $e_1$, with all spins on the path equal to $+1$. Define
\begin{equation}\label{eq:crossing-observable}
 f_q(\sigma_{Q_0})
 :=\widehat\Pp^{\,+}_{\beta,h}
   \bigl[C_q^+\mid\sigma_{Q_0}\bigr]\in[0,1].
\end{equation}
By Proposition~\ref{prop:conditioned-DLR}, the conditional law of the
bonds in $E(Q_0)$ depends only on $\sigma_{Q_0}$ and $\beta$. So $f_q$
is independent of $h$.
Recall $h_q$ from \eqref{eq:h0-choice} and
$\delta_q$ from \eqref{eq:delta-choice}.

\begin{lemma}\label{lem:crossing-observable}
The function $f_q$ is increasing in $\sigma_{Q_0}$. For every sufficiently large odd $q$,
\begin{equation}\label{eq:crossing-bounds}
 \langle f_q\rangle_{\mu_h^+}\ge1-2\delta_q,
 \quad \forall |h|\le h_q,
 \qquad
 \langle f_q\rangle_{\mu_0^-}\le\epsilon_q,
\end{equation}
where $\epsilon_q\to0$ as $q\to\infty$.
\end{lemma}

\begin{proof}
Given the spins, sample a bond configuration in $\{0,1\}^{E(Q_0)}$ as
\[
 \omega_{uv}=b_{uv}\ind_{\{\sigma_u=\sigma_v\}},
 \qquad b_{uv}\text{'s} \sim \text{ independent Bernoulli}(p).
\]
For fixed $(b_{uv})$, increasing the spins preserves every positive
open path. Averaging over these variables proves the monotonicity.

On $\{0\in\cG_\infty,A_+\}$, the crossing cluster of $Q_0$ belongs
to the positive FK mesh cluster, so $C_q^+$ occurs. Therefore, \eqref{eq:delta-choice} and \eqref{eq:sign-half} give
\begin{equation}\label{eq:crossinglowerbound}
	\langle f_q\rangle_{\widehat\mu^+_{\beta,h}}
	=\widehat\Pp^{\,+}_{\beta,h}[C_q^+] \ge 1-\frac{\Pp_{\beta,h}[0\notin\cG_\infty, A_+]}
	{\Pp_{\beta,h}[A_+]}
	\ge1-2\delta_q.
\end{equation}
 
Stochastic ordering of the ferromagnetic Ising measures gives $\widehat\mu^+_{\beta,h}\preceq\mu_h^+$. So
monotonicity of $f_q$ proves the first bound.

For the second bound, let $\Pp_{\beta,0}^-$ be the zero-field infinite-volume Edwards--Sokal coupling measure with spin marginal $\mu_0^-$ and bond marginal the infinite-volume wired zero-field FK measure $\phi_\beta^{\mathrm w}$. Conditional on $\sigma_{Q_0}$, the
bonds in $E(Q_0)$ are again independent, with open probabilities
$p\ind_{\{\sigma_u=\sigma_v\}}$.
Hence
\[
 f_q(\sigma_{Q_0})
 =\Pp_{\beta,0}^-\bigl[C_q^+\mid\sigma_{Q_0}\bigr],
 \qquad
 \langle f_q\rangle_{\mu_0^-}=\Pp_{\beta,0}^-[C_q^+].
\]
Since the infinite FK cluster under $\Pp_{\beta,0}^-$  has sign $-$, every FK
crossing with positive sign lies in a finite cluster.
Write $\mathcal C(x)$ for the FK cluster of $x$ and
$\operatorname{rad}_x\mathcal C(x):=\sup_{y\in\mathcal C(x)}\|y-x\|_\infty$.
There exist $c=c(d,\beta), C=C(d,\beta) \in (0,\infty)$ such that
\[
 \phi_\beta^{\mathrm w}
 \bigl[n\le\operatorname{rad}_x\mathcal C(x)<\infty\bigr]
 \le Ce^{-cn},\qquad n\ge1.
\]
For $d\ge3$, this follows from \cite[Theorem~5.104]{Gri06} and \cite{Bod05}; for $d=2$, it follows from planar duality and subcritical exponential decay, see, e.g., \cite[Chapter~6]{Gri06}. Every such crossing starts at some $x$ on the face
$F_q^-:=\{x\in Q_0:x_1=-q\}$, and its cluster has radius at least $2q$ from $x$. A union bound therefore gives
\[
 \langle f_q\rangle_{\mu_0^-}
 \le\sum_{x\in F_q^-}\phi_\beta^{\mathrm w}
 \bigl[2q\le\operatorname{rad}_x\mathcal C(x)<\infty\bigr]
 \le C(2q+1)^{d-1}e^{-2cq}
 =:\epsilon_q\longrightarrow0.
\]
This finishes the proof of the lemma.
\end{proof}

\subsection{Identification of the zero-field limit}
Although one can derive Proposition~\ref{prop:main-FM} and Theorem~\ref{thm:main-AF} directly from \eqref{eq:crossinglowerbound} and the same-field negative-state bound, we first prove the following continuity result, which is of independent interest.

\begin{proposition}\label{prop:continuity}
	Let $d\ge2$ and $\beta>\beta_c(d)$. Then
	\begin{equation}\label{eq:local-continuity}
		\mu^{\FM,\pm}_{\beta,h}\Rightarrow\mu^{\FM,\pm}_{\beta,0}, \ \mu^{\AF,\mathrm N\pm}_{\beta,h}\Rightarrow\mu^{\AF,\mathrm N\pm}_{\beta,0}
		\qquad\text{as }h\to0,
	\end{equation}
where $\Rightarrow$ denotes weak convergence.
\end{proposition}
\begin{proof}
By compactness, it suffices to identify any subsequential limit
$\mu_{h_n}^+\Rightarrow\nu$ with $h_n\to0$. By the DLR equations for $\mu_{h_n}^+$, for each $\Lambda \Subset \Z^d$ and local event $B$,
\begin{equation}\label{eq:DLRmu_h}
	\begin{aligned}
		\mu_{h_n}^+[B]&= \int \mu_{\Lambda,\beta,h_n}^{\FM,\eta}\bigl[B\bigr]\mu_{h_n}^+[d\eta]\\
		&=\int \left(\mu_{\Lambda,\beta,h_n}^{\FM,\eta}\bigl[B\bigr]-\mu_{\Lambda,\beta,0}^{\FM,\eta}\bigl[B\bigr]\right) \mu_{h_n}^+(d\eta) + \int \mu_{\Lambda,\beta,0}^{\FM,\eta}\bigl[B\bigr] \mu_{h_n}^+(d\eta);
	\end{aligned}
\end{equation}
here again $\mu_{\Lambda,\beta,h}^{\FM,\eta}$ is understood as its kernel. Since 
\[\sup_{\eta}\left |\mu_{\Lambda,\beta,h_n}^{\FM,\eta}\bigl[B\bigr]-\mu_{\Lambda,\beta,0}^{\FM,\eta}\bigl[B\bigr]\right | \rightarrow 0 \text{ as } n\rightarrow \infty,\]
we send $n\rightarrow\infty$ in \eqref{eq:DLRmu_h} to obtain
\[\nu[B]=\int \mu_{\Lambda,\beta,0}^{\FM,\eta}\bigl[B\bigr] \nu[d\eta].\]
These are the DLR equations for $\nu$. So $\nu$ is a zero-field Gibbs measure.

Write $T_z$ for translation by $z$. It acts on spin configurations by $(T_z\sigma)_u=\sigma_{u+z}$ and on measures by pushforward. Each $\mu_h^+$ is invariant under translations by $z$ with $|z|$ even, and $T_{e_1}\mu_h^+=\mu_{-h}^+$.
Then we have $\mu_{-h_n}^+ \Rightarrow T_{e_1} \nu$, and
\[
 \overline\nu:=\tfrac12(\nu+T_{e_1}\nu)
\]
is a translation-invariant zero-field Gibbs measure. The classification
of such measures gives
\begin{equation}\label{eq:mixture}
 \overline\nu=\alpha\mu_0^++(1-\alpha)\mu_0^-,
 \qquad \alpha\in[0,1],
\end{equation}
by the Aizenman--Higuchi theorem \cite{Aiz80,Hig81} for $d=2$ and Bodineau's theorem \cite{Bod06} for $d\ge3$.

Fix a sufficiently large odd $q$. We have $|h_n|\le h_q$ for all sufficiently large $n$.
Applying Lemma~\ref{lem:crossing-observable} to both $h_n$ and $-h_n$,
and using locality of $f_q$ to pass to the limit $h_n \rightarrow 0$, give
\[
 1-2\delta_q
 \le\langle f_q\rangle_{\overline\nu}
 \le\alpha+(1-\alpha)\epsilon_q.
\]
Letting $q\to\infty$ now gives $\alpha=1$. Since $\mu_0^+$ is extremal
among all Gibbs measures \cite[Lemma 6.65]{FV18},
$\overline\nu=\mu_0^+$ implies $\nu=\mu_0^+$.

Therefore $\mu_h^+\Rightarrow\mu_0^+$ as $h\rightarrow0$. Since a global spin flip maps
$\mu_{-h}^+$ to $\mu_h^-$, we have $\mu_h^-\Rightarrow\mu_0^-$ as $h\rightarrow0$.
The bipartite spin transformation then gives $\mu^{\AF,\mathrm N\pm}_{\beta,h}\Rightarrow\mu^{\AF,\mathrm N\pm}_{\beta,0}$.
\end{proof}

We are ready to prove our main result.
\begin{proof}[Proof of Proposition~\ref{prop:main-FM}
and Theorem~\ref{thm:main-AF}]
Proposition \ref{prop:continuity} implies that
\[\lim_{h\rightarrow 0} \frac{\langle \sigma_0 \rangle^{\FM,+}_{\beta,h}+\langle \sigma_{e_1} \rangle^{\FM,+}_{\beta,h}}{2}=\langle \sigma_0 \rangle^{\FM,+}_{\beta,0}>0.\]
So there is $h_0=h_0(d,\beta)>0$ such that
\[\frac{\langle \sigma_0 \rangle^{\FM,+}_{\beta,h}+\langle \sigma_{e_1} \rangle^{\FM,+}_{\beta,h}}{2}\geq \frac{\langle \sigma_0 \rangle^{\FM,+}_{\beta,0}}{2}=:c>0, \qquad \forall |h|\leq h_0. \]
Combining this with \eqref{eq:m-symmetry}, we obtain
\[\frac{\langle \sigma_0 \rangle^{\FM,+}_{\beta,h}-\langle \sigma_0 \rangle^{\FM,-}_{\beta,h}}{2}= \mu_{\beta,h}^{\FM,+}\bigl[\sigma_0=+1\bigr] - \mu_{\beta,h}^{\FM,-}\bigl[\sigma_0=+1\bigr] \geq c, \qquad \forall |h|\leq h_0. \]
Monotonicity in the boundary condition implies that for each $n\geq 1$,
\[\mu^{\FM,+}_{B_n,\beta,h}\bigl[\sigma_0=+1\bigr]
-\mu^{\FM,-}_{B_n,\beta,h}\bigl[\sigma_0=+1\bigr] \geq \mu_{\beta,h}^{\FM,+}\bigl[\sigma_0=+1\bigr] - \mu_{\beta,h}^{\FM,-}\bigl[\sigma_0=+1\bigr] \geq c, \quad |h|\leq h_0.\]
This completes the proof of Proposition~\ref{prop:main-FM}
and thus Theorem~\ref{thm:main-AF}.
\end{proof}

\section*{Acknowledgments}
This research was partially supported by the National Natural Science Foundation of China (No. 12271284 and No. 12226001). The first author is grateful to Elliott Lieb and Chuck Newman for inspiring this line of research and for initial discussions. The authors acknowledge the use of ChatGPT as an auxiliary tool in proof development, including synthesizing and completing candidate arguments from ideas and mathematical ingredients supplied by the authors, identifying potentially relevant literature, and improving the presentation. The authors are responsible for the content.

\bibliographystyle{abbrv}
\bibliography{references}

\end{document}